\documentclass[11pt]{amsart}
\usepackage{a4wide,color}%,refcheck}

\usepackage{amsfonts, amsmath, amscd}
\usepackage[psamsfonts]{amssymb}
\usepackage{amssymb,textcomp,lmodern}
\usepackage{pb-diagram}
\usepackage[all,cmtip]{xy}
\newtheorem{theorem}{Theorem}[section]
\newtheorem{lemma}[theorem]{Lemma}
\newtheorem{corollary}[theorem]{Corollary}

\newtheorem{proposition}[theorem]{Proposition}

\newtheorem{example}[theorem]{Example}

\newtheorem{problem}[theorem]{Problem}

\newtheorem{claim}[theorem]{Claim}

\theoremstyle{definition}
\newtheorem{definition}[theorem]{Definition}

\numberwithin{equation}{section}

\newcommand{\e}{\varepsilon}
\newcommand{\w}{\omega}

\newcommand{\ID}{\mathbb{D}}

\newcommand{\IR}{\mathbb{R}}

\newcommand{\Tau}{\mathcal T}

\newcommand{\IF}{\mathbb{F}}

\newcommand{\TTT}{\mathcal{T}}

\newcommand{\F}{\mathcal{F}}

\newcommand{\KK}{\mathcal{K}}

\newcommand{\diam}{\mathsf{diam}}

\newcommand{\BNP}{\mathsf{BNP}}

\newcommand{\CC}{C_k}

\newcommand{\SM}{{\setminus}}

\input xy
\xyoption{all}

\title[The Gelfand--Phillips property for locally convex spaces]{The Gelfand--Phillips property for locally convex spaces}
\author{Taras Banakh and Saak Gabriyelyan}

\address{Ivan Franko National University of Lviv (Ukraine) and Jan Kochanowski University in Kielce (Poland)}
\email{t.o.banakh@gmail.com}

\address{Department of Mathematics, Ben-Gurion University of the Negev, Beer-Sheva, P.O. 653, Israel}
\email{saak@math.bgu.ac.il}



\begin{document}

\begin{abstract}We extend the well-known Gelfand--Phillips property for Banach spaces to locally convex spaces, defining a locally convex space $E$ to be Gelfand--Phillips if  every limited set in $E$ is precompact in the topology on $E$ defined by barrels.
Several characterizations of Gelfand--Phillips spaces are given.
The problem of preservation of the Gelfand-Phillips property by standard operations over locally convex spaces is considered. Also we explore the Gelfand--Phillips property in spaces $C(X)$ of continuous functions on a Tychonoff space $X$.
If $\tau$ and $\TTT$ are two  locally convex topologies on $C(X)$ such that $\TTT_p\subseteq \tau\subseteq \TTT\subseteq \TTT_k$, where $\TTT_p$ is the topology of pointwise convergence and $\TTT_k$ is the compact-open topology on $C(X)$, then the Gelfand--Phillips property of the function space $(C(X),\tau)$ implies the Gelfand--Phillips property of $(C(X),\TTT)$. If additionally $X$ is metrizable, then the function space $\big(C(X),\TTT\big)$  is Gelfand--Phillips.
\end{abstract}

\maketitle

%%%%%%%%%%%%%%%%%%%%%%%%%%%%%%%%%%%%%%%%%%%%%%%
%%%%%%%%%%%%%%%%%%%%%%%%%%%%%%%%%%%%%%%%%%%%%%%
%%%%%%%%%%%%%%%%%%%%%%%%%%%%%%%%%%%%%%%%%%%%%%%

\section{Introduction}

%%%%%%%%%%%%%%%%%%%%%%%%%%%%%%%%%%%%%%%%%%%%%%%
%%%%%%%%%%%%%%%%%%%%%%%%%%%%%%%%%%%%%%%%%%%%%%%
%%%%%%%%%%%%%%%%%%%%%%%%%%%%%%%%%%%%%%%%%%%%%%%

All locally convex spaces  are assumed to be Hausdorff, infinite-dimensional and over the field $\IF$ of real or complex numbers, and all topological spaces are assumed to be infinite and Tychonoff. We denote by $E'$ the topological dual of an lcs $E$. The dual space $E'$ of $E$ endowed with the weak$^\ast$ topology $\sigma(E',E)$ and the strong topology $\beta(E',E)$ is denoted by $E'_{w^\ast}$ and $E'_\beta$, respectively.  For a bounded subset $B\subseteq E$ and a functional $\chi\in E'$, we put
\[
\|\chi\|_B:= \sup\big\{ |\chi(x)|:x\in B\cup\{0\}\big\}.
\]

Let $E$ be a Banach space. The closed unit ball of $E$ is denoted by $B_E$. A bounded subset $B$ of $E$ is called {\em limited} if each weak$^\ast$ null sequence $\{ \chi_n\}_{n\in\w}$ in $E'$ converges uniformly on $B$, that is $\lim_{n\to\infty} \|\chi_n\|_B =0$.
A Banach space $E$ is said to have the {\em Gelfand--Phillips property} ({\em $(GP)$ property} for short) or is a {\em Gelfand--Phillips space} if every limited set in $E$ is precompact.%; in this case we shall write $E\in (GP)$.

The classical result of  Gelfand \cite{Gelfand} states that  every separable Banach space is Gelfand--Phillips. On the other hand, Phillips \cite{Phillips} showed that the non-separable Banach space $\ell_\infty=C(\beta\w)$ is not Gelfand--Phillips, where $\beta\w$ is the Stone--\v{C}ech compactification of the discrete space $\w$ of nonnegative integers.

Following Bourgain and Diestel  \cite{BourDies}, a bounded linear operator $T$ from a Banach space $L$ into $E$ is called {\em limited} if $T(B_L)$ is a limited subset of $E$. %, or equivalently, if the adjoint operator $T^\ast$ is weak$^\ast$-norm sequentially continuous.
In \cite{Drewnowski}, Drewnowski noticed the next characterization of Banach spaces with the $(GP)$ property.

\begin{theorem} \label{t:Drew-GP}
For a Banach space $E$ the following assertions are equivalent:
\begin{enumerate}
\item[{\rm (i)}] $E$ is Gelfand--Phillips;
\item[{\rm (ii)}] every limited weakly null sequence in $E$ is norm null;
\item[{\rm (iii)}] every limited operator with range in $E$ is compact.
\end{enumerate}
\end{theorem}
\noindent This characterization of Gelfand--Phillips Banach spaces plays a crucial role in many arguments for establishing the $(GP)$  property in Banach spaces. % (since the implication (ii)$\Rightarrow$(i) is not trivial we give its detailed proof and generalization in Theorem~\ref{t:GP-weak-null}).
The Gelfand--Phillips property was intensively studied in particular in \cite{CGP,Drewnowski,DrewEm,Schlumprecht-C}. It follows from results of Schlumprecht \cite{Schlumprecht-Ph,Schlumprecht-C} that the $(GP)$ property is not a three space property (see also Theorem 6.8.h in \cite{CG}).
In our recent paper \cite{BG-GP-Banach}, we give several new characterizations of  Gelfand--Phillips Banach spaces.

Another direction for studying the  Gelfand--Phillips property is to characterize Gelfand--Phillips spaces that belong to some important classes of Banach spaces. In the next proposition, whose short proof is given in Corollary~2.2 of \cite{BG-GP-Banach}, we provide some of the most important and general results in this direction (for all relevant definitions see Section \ref{sec:GP}).

\begin{theorem} \label{t:Banach-GP}
A Banach space $E$ is Gelfand--Phillips if  one of the following conditions holds:
\begin{enumerate}
\item[{\rm (i)}] {\rm(\cite[Cor.~2.2]{BG-GP-Banach})} the closed unit ball  $B_{E'}$ of the dual space $E'$ endowed with the weak$^\ast$ topology is selectively sequentially pseudocompact;%$($in particular, Valdivia$)$ compact;
\item[{\rm (ii)}]  {\rm(\cite{Gelfand})} $E$ is separable;
\item[{\rm (iii)}] {\rm(\cite[Prop.~2]{CGP})} $E$ is  separably weak$^\ast$-extensible \textup{(}$\Leftrightarrow E$  has the separable $c_0$-extension property$)$;
\item[{\rm (iv)}]  {\rm(cf. \cite[Th.~2.2]{Drewnowski} and \cite[Prop.~2]{Schlumprecht-C})} the space $E'_{w^\ast}$ is selectively sequentially pseudocompact at some $E$-norming set $S\subseteq E'$;
\item[{\rm(v)}]  {\rm(\cite[Th.~4.1]{DrewEm})} $E=C(K)$ for some  compact selectively sequentially pseudocompact space $K$.
%\item[{\rm (vi)}] {\rm(\cite[Cor.~2.2]{BG-GP-Banach})} $E=C(K)$ for some compact space $K$ such that $P(K)$ is  selectively sequentially pseudocompact at some set $A\subseteq P(K)$ containing $K$.
\end{enumerate}
\end{theorem}

The aforementioned results and discussion suggest the problem of defining and studying the Gelfand--Phillips property in the class of all locally convex spaces which  is the main purpose of the article. To this end, instead of the original topology on a locally convex space $E$ we consider the topology $\beta(E,E')$ on $E$ whose neighborhood base at zero consists of barrels (for more details about this topology, see \S~8.4 of \cite{Jar}). We shall say that a subset $A$ of $E$ is {\em barrel-bounded}  or {\em barrel-precompact} if $A$ is  bounded or, respectively, precompact in the space $\big(E,\beta(E,E')\big)$. 

\begin{definition} \label{def:limited-lcs}
A barrel-bounded subset $A$ of a locally convex space $E$ is called {\em limited} if every weak$^\ast$ null sequence $\{\chi_n\}_{n\in\w}$ in $E'$ converges uniformly on $A$, that is,
$
\lim_{n\to\infty} \|\chi_n\|_A=0.%\mbox{\qed}
$ \qed
\end{definition}

It is clear that if $E$ is a Banach space, then a subset $A$ of $E$ is limited if and only if it is limited in the usual sense. Now Gelfand--Phillips spaces are defined in a very natural way as follows.

\begin{definition} \label{def:GP-lcs}
A locally convex space $E$ is said to have {\em the Gelfand--Phillips property} or else $E$ is a {\em Gelfand--Phillips space} if every limited subset of $E$ is barrel-precompact.\qed
\end{definition}
It is easy to see that a Banach space $E$ has the Gelfand--Phillips property if and only if it is Gelfand--Phillips in the sense of Definition \ref{def:GP-lcs}.
\smallskip

Now we describe the content of the article.
In Section \ref{sec:GP} we give several characterizations of Gelfand--Phillips spaces generalizing and extending Theorem \ref{t:Drew-GP} and characterizations of Gelfand--Phillips Banach spaces obtained in \cite{BG-GP-Banach}, see Theorems \ref{t:GP-characterization} and \ref{t:GP-weak-null}.
In Corollary \ref{c:GP-sufficient} we generalize (i) and (ii) of Theorem \ref{t:Banach-GP}, and Corollary  \ref{c:barrelled-exten-GP} generalizes (iii) of Theorem \ref{t:Banach-GP}. In Theorem \ref{t:not-eJNP} we give a sufficient condition on a locally convex space implying the failure of the $(GP)$ property.

In Section \ref{sec:GP-f} we study the  Gelfand--Phillips  property in function spaces $C(X)$ endowed with different locally convex topologies $\TTT$, but mainly  we consider the pointwise topology $\TTT_p$ and the compact-open topology $\TTT_k$. In Corollary \ref{c:GP-sspC} of Theorem \ref{t:Cp-GP}, we complement (v)  of Theorem \ref{t:Banach-GP} by showing  that
for every selectively sequentially pseudocompact space $K$, the function space $C_p(X)$ is Gelfand--Phillips. In Theorem \ref{t:GP-between} we show that for every Tychonoff space $X$,  if $C_p(X)$ is Gelfand--Phillips then so is $\CC(X)$. We prove that for every metrizable space $X$, the function spaces $C_p(X) $ and $\CC(X)$ are Gelfand--Phillips, see Corollary \ref{c:mu-Cp-Ck-GP}.

%As a corollary of Theorem \ref{t:GP-characterization} we show in Corollary \ref{c:GP-subspace} that a barrelled subspace of a Gelfand--Phillips space has the $(GP)$ property. This result generalizes the well-known fact that a closed subspace of a Gelfand--Phillips Banach space is also Gelfand--Phillips. Corollary \ref{c:GP-sufficient} provides some sufficient conditions on a locally convex space $E$ to have the Gelfand--Phillips property. In particular, it is shown that every separable and barrelled space $E$ is  Gelfand--Phillips generalizing the above-mentioned result of Gelfand \cite{Gelfand} for Banach spaces. In \cite{CGP}, Castillo, Gonz\'{a}lez and Papini proved that every Banach space with the separable $c_0$-extension property is Gelfand--Phillips. We extend this result to all barrelled spaces, see Corollary \ref{c:barrelled-exten-GP}.

%%%%%%%%%%%%%%%%%%%%%%%%%%%%%%%%%%%%%%%%%%%%%%%%
%%%%%%%%%%%%%%%%%%%%%%%%%%%%%%%%%%%%%%%%%%%%%%%%
%%%%%%%%%%%%%%%%%%%%%%%%%%%%%%%%%%%%%%%%%%%%%%%%
%%%%%%%%%%%%%%%%%%%%%%%%%%%%%%%%%%%%%%%%%%%%%%%%

\section{A characterization of Gelfand--Phillips locally convex spaces} \label{sec:GP}

%%%%%%%%%%%%%%%%%%%%%%%%%%%%%%%%%%%%%%%%%%%%%%%%
%%%%%%%%%%%%%%%%%%%%%%%%%%%%%%%%%%%%%%%%%%%%%%%%
%%%%%%%%%%%%%%%%%%%%%%%%%%%%%%%%%%%%%%%%%%%%%%%%
%%%%%%%%%%%%%%%%%%%%%%%%%%%%%%%%%%%%%%%%%%%%%%%%

In this section we characterize locally convex spaces with the  Gelfand--Phillips  property. But first we recall some basic definitions.

Let $E$ be a locally convex space (lcs for short). We denote  by $E_\beta$ the space $E$ endowed with the locally convex topology $\beta(E,E')$ whose neighborhood base at zero consists of barrels.

A subset $A$ of an lcs $E$ is called {\em precompact} if for every neighborhood $U$ of zero there is a finite subset $F\subseteq E$ (one can take $F\subseteq A$) such that $A\subseteq F+U$. Therefore $A$ is barrel-precompact if it is a precompact subset of $E_\beta$. We say that a subset $B\subseteq E$ is {\em barrel-bounded} if it is a bounded subset of $E_\beta$, that is, if for any barrel $U\subseteq E$ there is an $n\in\w$ such that $B\subseteq nU$. A subset $A$ of $E$ is defined to be {\em barrel-separated} if there exists a barrel $B\subseteq E$ such that $A$ is {\em $B$-separated} in the sense that $a-a'\notin B$ for any distinct elements $a,a'\in A$.
It is easy to see that a subset $A\subseteq B$ is {\em not} barrel-precompact if and only if it contains an infinite barrel-separated subset. The {\em polar} $D^\circ$ of a set $D\subseteq E$ is the set
\[
D^\circ:=\{ \chi\in E': \|\chi\|_D\leq 1\}.
\]

To characterize locally convex spaces with the $(GP)$ property, we need the following topological notions. Let $A$ be a subset of a topological space $X$.
Then $A$ is defined to be {\em relatively sequentially compact} in $X$ if every sequence $\{x_n\}_{n\in\w}\subseteq A$ contains a subsequence $\{x_{n_k}\}_{k\in\w}$ that converges in $X$. Following \cite{BG-GP-Banach}, we define  the space $X$ to be {\em selectively sequentially pseudocompact} at $A$ if for any open sets $U_n\subseteq X$, $n\in\w$, intersecting the set $A$, there exists a sequence $(x_n)_{n\in\w}\in\prod_{n\in\w}U_n$ containing a subsequence $(x_{n_k})_{k\in\w}$ that converges in $X$.
It is clear that if $A$ is a relatively sequentially compact subset of $X$, then $X$ is selectively sequentially pseudocompact at $A$.

A  topological vector space  $X$ is defined to be  {\em selectively sequentially precompact} at $A$ if for any open sets $U_n\subseteq X$, $n\in\w$, intersecting the set $A$ there exists a sequence $(x_n)_{n\in\w}\in\prod_{n\in\w}U_n$ containing a Cauchy subsequence. We recall that a sequence $(x_n)_{n\in\w}$ in $X$ is {\em Cauchy} if for any neighborhood $U\subseteq X$ of zero  there exists $n\in\w$ such that $x_m-x_k\in U$ for every $m,k\ge n$.
It is clear that if $X$ is selectively sequentially pseudocompact at $A$, the $X$ is also selectively sequentially precompact at $A$.

Analogously to limited and compact operators between Banach spaces we define
\begin{definition} \label{def:limited-oper}
An operator $T:X\to Y$ between locally convex spaces is called
\begin{enumerate}
\item[{\rm (i)}] {\em limited} if  for every barrel-bounded subset of $X$, the image $T(B)$ is a limited subset of $Y$;
\item[{\rm (ii)}]   {\em barrel-precompact} if  for every barrel-bounded subset $B$ of $X$, the image $T(B)$ is barrel-precompact;
\item[{\rm (iii)}] {\em $\beta$-to-$\beta$ precompact} if the operator $T: X_\beta \to Y$ is barrel-precompact.
\end{enumerate}
\end{definition}

%Let $D$ be an absolutely convex, closed and bounded subset of a locally convex space $E$. Denote by $E_D$ the linear span of $D$ in $E$. Then the gauge $p_D$ of $D$ in $E_D$ defines a norm on the linear span $\spn(D)$ of $D$ in $E$. The normed space $(\spn(D),p_D)$ is denoted by $E_D$. Observe that the identity inclusion $E_D\to E$ is continuous.
 The following characterization of locally convex spaces with the $(GP)$ property is the main result of this section.

\begin{theorem} \label{t:GP-characterization}
For a locally convex space $E$ the following assertions are equivalent:
\begin{enumerate}
\item[{\rm (i)}] $E$ has the Gelfand--Phillips property; %Every limited subset of $E$
\item[{\rm (ii)}] For every barrel-bounded set $B\subseteq E$ which is not  barrel-precompact, there is a weak$^\ast$ null sequence $\{\chi_n\}_{n\in\w}$ in $E'$ such that $\|\chi_n\|_B \not\to 0$.
\item[{\rm (iii)}] For any infinite barrel-bounded, barrel-separated subset $D$ of $E$ and every $\delta>0$ there exist a sequence $\{x_n\}_{n\in\w}$ in $D$ and a null sequence $\{f_n\}_{n\in\w}$ in $E'_{w^\ast}$ such that $|f_n(x_k)|<\delta$ and $|f_n(x_n)|>1+\delta$ for all natural numbers $k<n$.
\item[{\rm (iv)}] For any infinite barrel-bounded barrel-separated set $D$ in $E$ there exists a continuous operator $T:E\to C_p^0(\w)$ such that $T(D)$ is not  precompact in the Banach space $c_0$.
\item[{\rm (v)}] For each infinite barrel-bounded barrel-separated  subset $D$ of $E$ there exist an infinite   subset $X \subseteq D$ and a subset $S\subseteq E'_{w^\ast}$ such that $E'_{w^\ast}$ is selectively sequentially precompact at $S$  and $\sup_{s\in S}|s(x-x')|>1$ for every distinct $x,x'\in X$.
\item[{\rm (vi)}] Every limited operator $T:L\to E$ from an lcs space $L$ to $E$ is barrel-precompact.
\item[{\rm (vii)}] Every limited operator $T:L\to E$ from a normed space $L$ to $E$ is barrel-precompact.
\end{enumerate}
%Moreover,  the {\rm (i)--(vii)} imply the following condition
%\begin{enumerate}
%\item[{\rm (viii)}] Every $\beta(E,E')$-weakly null sequence in $E$ which is limited is barrel-precompact.
%\end{enumerate}
\end{theorem}

\begin{proof}
(i)$\Rightarrow$(ii) Fix any barrel-bounded set $B\subseteq E$ which is not barrel-precompact. By (i), $E$ is Gelfand--Phillips and hence $B$ is not limited. Therefore there exists a weak$^\ast$ null sequence $\{\chi_n\}_{n\in\w}$ in $E'$ such that $\|\chi_n\|_B\not\to 0$, as desired.
\smallskip

(ii)$\Rightarrow$(iii) Fix any  $\delta>0$ and any infinite barrel-bounded barrel-separated set $D$ in  $E$. Find a  barrel $B\subseteq E$ such that the set $D$ is  $B$-separated. Observe that $D$ is not barrel-precompact because $D$ is $B$-separated and infinite. By (ii), there is a null sequence $\{g_n\}_{n\in\w}$ in $E'_{{w^\ast}}$ such that $\|g_n\|_D \not\to 0$. Passing to a subsequence and multiplying each functional $g_n$ by an appropriate constant if needed, we can assume that
\begin{equation} \label{equ:eJNP-2}
\|g_n\|_D = \sup_{d\in D}|g_n(d)| >1+\delta \quad \mbox{ for every }\; n\in\w.
\end{equation}
Now, choose arbitrarily $x_0\in D$ such that $|g_0(x_0)|>1+\delta$. Assume that, for $k\in\w$, we found $x_0,\dots,x_k\in D$ and a sequence $0=n_0 <n_1<\cdots<n_k$ of natural numbers such that
\[
|g_{n_j}(x_i)|<\delta \; \mbox{ and } \; |g_{n_j}(x_j)| >1+\delta \quad \mbox{ for every } \; 0\leq i<j\le k.
\]
Since $g_n\to 0$ in the topology $\sigma(E',E)$, (\ref{equ:eJNP-2}) implies that there are $n_{k+1}>n_k$ and $x_{k+1} \in D\SM \{x_0,\dots,x_k\}$ such that  $|g_{n_{k+1}}(x_{k+1})| >1+\delta$ and $|g_{n_{k+1}}(x_i)|<\delta$ for every $i\le k$. For every $k\in\w$, put $f_k:=g_{n_k}$, and observe that $f_k\to 0$ in  $E'_{w^\ast}$ and
\[
|f_k(x_i)|=|g_{n_k}(x_i)|<\delta \; \mbox{ and }\; |f_k(x_k)|=|g_{n_k}(x_k)|>1+\delta
\]
for any numbers $i<k$, as desired.
\smallskip

(iii)$\Rightarrow$(iv)  Let $D$ be any barrel-bounded barrel-separated set in $E$. By (iii), for $\delta=\tfrac{1}{2}$ there exist a sequence $\{x_n\}_{n\in\w}$ in $D$ and a null sequence $\{\chi_n\}_{n\in\w}$ in $E'_{{w^\ast}}$ such that $|\chi_m(x_n)|<\frac{1}{2}$ and $|\chi_m(x_m)|>\frac{3}{2}$ for all $n<m$. Then $|\chi_m(x_m)-\chi_m(x_n)|\ge |\chi_m(x_m)|-|\chi_m(x_n)|>\frac{3}{2}-\frac{1}{2}=1$ for every $n<m$. It follows that the operator
\[
T: E\to  C_p^0(\w), \quad T(x):= \big(\chi_n(x)\big)_{n\in\w},
\]
is well-defined and continuous. Since $\|T(x_n)-T(x_m)\|_{\infty}\ge |\chi_m(x_m)-\chi_m(x_n)|>1$ for every $n<m$, the set $T(D)\supseteq\{T(x_n)\}_{n\in\w}$ is not precompact in $c_0$.
\smallskip

(iv)$\Rightarrow$(i) Fix a barrel-bounded subset $P\subseteq E$, which is not barrel-precompact. Then there exists a barrel $B\subseteq E$ such that $L\not\subseteq F+B$ for any finite subset $F\subseteq E$. For every $n\in\w$, choose inductively a point $z_n\in P$ so that $z_n\notin \bigcup_{k<n}(z_k+B)$. Observe that the  set  $D=\{z_n:n\in\w\}$ is infinite, barrel-bounded and barrel-separated. By (iv), there exists a continuous operator $T:E\to C_p^0(\w)$ such that the set $T(D)$ is not precompact in the Banach space $c_0$.
 Since $c_0$ is a Banach space, there exist a sequence $\{ a_n\}_{n\in\w}$ in $D$ and $\delta>0$ such that  $\| T(a_n)-T(a_m)\|_{c_0}\geq \delta$ for all distinct $n,m\in\w$.
Observe that the sequence $\{ T(a_n)\}_{n\in\w}$ is bounded in the Banach space $c_0$. Therefore there are two sequences $0\leq n_0<n_1<\cdots$ and $0\leq m_0< m_1<\cdots$ of natural numbers such that
\[
\big| e'_{m_k}\big( T(a_{n_k})\big)\big| >\tfrac{\delta}{2} \quad \mbox{ for every }\; k\in\w,
\]
where $e'_n: C^0_p(\w)\to \IF$ is the $n$th coordinate functional. For every $k\in\w$, set  $f_k :=e'_{m_k} \circ T$. It follows that $\{f_k\}_{k\in\w}$ is a null sequence in $E'_{w^\ast}$ and
\[
\|f_k\|_P =\sup_{x\in P} |f_k(x)| \geq |f_k (a_{n_k})|>\tfrac{\delta}{2}
\]
for every $ k\in\w$, witnessing that the set $P$ is not limited.
\smallskip

(iii)$\Rightarrow$(v) By (iii), for any infinite barrel-bounded, barrel-separated set $X\subseteq E$ there exists a null sequence $S=\{s_n\}_{n\in\w}\subseteq E'_{w^\ast}$ and a subsequence $X_0=\{x_n\}_{n\in\w}\subseteq X$ such that $|s_n(x_n)|>2$ and $|s_n(x_k)|<1$ for all $k<n$. Taking into account that $S$ is a null sequence in $E'_{w^\ast}$, we conclude that $S$ is relatively sequentially compact in $E'_{w^\ast}$ and hence $E'_{w^\ast}$ is selectively sequentially precompact at $S$. Finally, observe that for any $k<n$
\[
\sup_{s\in S}|s(x_n-x_k)|\ge |s_n(x_n)-s_n(x_k)|>2-1=1.
\]
\smallskip

(v)$\Rightarrow$(ii)  Fix any barrel-bounded subset $P\subseteq E$, which is not barrel-precompact. Then there exists a barrel $B\subseteq E$ such that $P\not\subseteq F+B$ for any finite subset $F\subseteq E$. For every $n\in\w$, choose inductively a point $z_n\in P$ so that $z_n\notin \bigcup_{k<n}(z_k+B)$. Observe that the  set  $\{z_n:n\in\w\}$ is infinite, barrel-bounded and barrel-separated. By (v), there is an infinite countable  subset $X \subseteq \{z_n:n\in\w\}$ and a subset $S\subseteq E'_{w^\ast}$ such that  $E'_{w^\ast}$ is selectively sequentially precompact at $S$ and
\[
\sup_{s\in S}|s(x-x')|>1 \; \mbox{ for every distinct } x,x'\in X.
\]
Write the set $X$ as $\{x_n\}_{n\in\w}$ for pairwise distinct points $x_n$.
Then for every $n<m$, we can find a linear functional $f_{n,m}\in S$ such that $|f_{n,m}(x_m-x_n)|>1$. It follows from the selective sequential precompactness of $E'_{w^\ast}$ at $S$ that the set $S$ is bounded in $E'_{w^\ast}$, see for example \cite[Lemma~2.9]{Gabr-free-resp}. Therefore, by Theorem 8.3.2 in \cite{Jar}, the polar $D:=S^\circ$ is a barrel in $E$.  Since the set $P$ is barrel-bounded, the real number $|P|_D=\inf\{r\ge 0:P\subseteq r\cdot D\}$ is well-defined.  It follows that  $X=\{x_n:n\in\w\}\subseteq P\subseteq |P|_D\cdot D$.  Consider the  disk $\ID=\{z\in\IF:|z|\le|P|_D\}$ in the field $\IF$. Observe that for every $n<m$ the inclusion $f_{n,m}\in S\subseteq S^{\circ\circ}= D^\circ$ implies
\[
f_{n,m}(X)\subseteq f_{n,m}(|P|_D\cdot D)=|P|_D\cdot f_{n,m}(D)\subseteq\ID,
\]
which means that $f_{n,m}{\restriction}_X\in\ID^X$.

%Let $T:=\{(n,m)\in\w\times \w:n<m\}$, $X:=\{x_n\}_{n\in\w}\subset E$ and $Y:=\{f_{n,m}{\restriction}_X\}_{(n,m)\in T}\subset \ID^X$. Let $\bar Y$ be the closure of $Y$ in the compact metrizable space $\ID^X$. For every $n\in\w$, let $p_{n}:\bar Y\to\IF$, $p_n:y\mapsto y(x_n)$, be the coordinate projection. Observe that
%$$|p_{n}(f_{n,m}{\restriction}_X)-p_{m}(f_{n,m}{\restriction}_X)|=|f_{n,m}(x_n)-f_{n,m}(x_m)|>1.$$

Using the sequential compactness of the compact metrizable space $\ID^X$, we can construct a decreasing sequence $\{\Omega_n\}_{n\in\w}$ of infinite sets in $\w$ such that for every $n\in\w$, the sequence $\{f_{n,m}{\restriction}_X\}_{m\in\Omega_n}$ converges to some function $f_n\in \ID^X$.

Choose an infinite set $\Omega\subseteq\w$ such that $\Omega\setminus\Omega_n$ is finite for every $n\in\w$. Since the space $\ID^X$ is sequentially compact, we can replace $\Omega$ by a smaller infinite set and additionally assume that the sequence $\{f_n\}_{n\in\Omega}$ converges to some element $f_\infty\in \ID^X$. Since the set $\{f_\infty(x_n)\}_{n\in\Omega}\subseteq\ID$ admits a finite cover by sets of diameter $<\frac14$, we can replace $\Omega$ by a suitable infinite subset and additionally assume that the set $\{f_\infty(x_n)\}_{n\in\Omega}$ has diameter $<\frac14$.

It follows that  the function $f_\infty\in\ID^X$ belongs to the closure of the set $\{f_{n,m}{\restriction}_X:n,m\in \Omega,\; n<m\}$.
Since the space $\ID^X$ is first-countable, we can choose a sequence $\{(n_i,m_i)\}_{i\in\w}\subseteq\{(n,m)\in\Omega\times\Omega:n<m\}$ such that the sequence $\{f_{n_i,m_i}{\restriction}_X\}_{i\in\w}$ converges to $f_\infty$. Since $\IF^X$ is metrizable, for every $i\in\w$ the element $f_{n_i,m_i}{\restriction}_X$ of $\ID^X\subseteq\IF^X$ has an open neighborhood $V_i\subseteq \IF^X$  such that the sequence $\{V_i\}_{i\in\w}$ converges to $f_\infty$ in the sense that each neighborhood of $f_\infty$ in $\IF^X$ contains all but finitely many sets $V_i$. Since $|f_{n_i,m_i}(x_{n_i}-x_{m_i})|>1$, we can replace each set $V_i$ by a smaller neighborhood of $f_{n_i,m_i}{\restriction}_X$ and additionally assume that $|g(x_{n_i})-g(x_{m_i})|>1$ for every $g\in V_i$. For every $i\in\w$ consider the open neighborhood $W_i:=\{f\in E'_{w^\ast}:f{\restriction}_X\in V_i\}$ of the functional $f_{n_i,m_i}$ in the space $E'_{w^\ast}$.

Since $E'_{w^\ast}$ is selectively sequentially precompact at $S$, there exists a Cauchy sequence $\{g_k\}_{k\in\w}\subseteq E'_{w^\ast}$ and an increasing number sequence $\{i(k)\}_{k\in\w}$ such that $g_k\in W_{i(k)}$ for every $k\in\w$. Since the sequence $(g_k)_{k\in\w}$ is Cauchy in $E'_{w^\ast}\subseteq \IF^E$, it converges to some linear functional $g_\infty\in\IF^E$ (which is not necessarily continuous on $E$). The continuity of the restriction operator $E'_{w^\ast}\to\IF^X$, $f\mapsto f{\restriction}_X$, and the choice of the open sets $V_i$, $i\in\w$, guarantee that $g_\infty{\restriction}_X=f_\infty$. Consequently, the sequence $(g_k{\restriction}_X)_{k\in\w}$ converges to $g_\infty{\restriction}_X =f_\infty$ in $\IF^X$.

Then for every $k\in \w$, we can find a number $j_k>k$ such that
\[
\max\{|f_\infty(x_{n_{i(k)}})-g_{j_k}(x_{n_{i(k)}})|, |f_\infty(x_{m_{i(k)}})-g_{j_k}(x_{m_{i(k)}})|\}<\tfrac{1}{8}.
\]

For every $k\in\w$, consider the functional $\mu_k:=g_k-g_{j_k}\in E'$ and observe that the sequence $\{\mu_k\}_{k\in\w}$ converges to zero in $E'_{w^\ast}$. On the other hand,  for every $k\in \w$, the choice of the sequence $(j_k)_{k\in\w}$,
the inequality $\diam \{f_\infty(x_n)\}_{n\in\Omega}<\frac14$, and the inclusion $g_k{\restriction}_X\in V_{i(k)}$  imply
\[
\begin{aligned}
|g_{j_k}(x_{n_{i(k)}}) & -g_{j_k}(x_{m_{i(k)}})|\\
& =|g_{j_k}(x_{n_{i(k)}})-f_\infty(x_{n_{i(k)}})+f_\infty(x_{m_{i(k)}})-g_{j_k}(x_{m_{i(k)}})+f_\infty(x_{n_{i(k)}})-f_\infty(x_{m_{i(k)}})|\\
& \le|g_{j_k}(x_{n_{i(k)}})-f_\infty(x_{n_{i(k)}})|+|f_\infty(x_{m_{i(k)}})-g_{j_k}(x_{m_{i(k)}})|+|f_\infty(x_{n_{i(k)}})-f_\infty(x_{m_{i(k)}})|\\
&\le \tfrac{1}{8}+\tfrac{1}{8}+\tfrac{1}{4}=\tfrac{1}{2}
\end{aligned}
\]
and
\[
\begin{aligned}
|\mu_k(x_{n_{i(k)}}) -\mu_k (x_{m_{i(k)}})|& =|g_{k}(x_{n_{i(k)}})-g_{j_k}(x_{n_{i(k)}})-g_k(x_{m_{i(k)}})+g_{j_k}(x_{m_{i(k)}})|\\
& \ge |g_{k}(x_{n_{i(k)}})-g_k(x_{m_{i(k)}})|-|g_{j_k}(x_{n_{i(k)}})-g_{j_k}(x_{m_{i(k)}})|>1-\tfrac{1}{2}=\tfrac{1}{2}.
\end{aligned}
\]
Then
\[
\sup_{x\in P}|\mu_k(x)|\ge \max\{|\mu_k(x_{n_{i(k)}})|,|\mu_k(x_{m_{i(k)}})|\}\ge\tfrac14,
\]
witnessing that $\|\mu_k\|_P\not\to 0.$%the space $E$ has the Gelfand--Phillips property.
\smallskip

(i)$\Rightarrow$(vi) Let $T:L\to E$ be a limited operator from an lcs $L$ to $E$, and let $B$ be a barrel-bounded subset of $L$. Then the image $T(B)$ is a limited subset of $E$ and hence, by (i), $T(B)$ is barrel-precompact. Thus $T$ is a barrel-precompact operator.
\smallskip

(vi)$\Rightarrow$(vii) is trivial.
\smallskip

(vii)$\Rightarrow$(i) Fix a limited subset $B$ of $E$. It is clear that the closed absolutely convex hull $D$ of the set $B$ in $E$ is also limited. Let $L$ be the linear hull of $D$. Since $D$ is (barrell)-bounded in $E$, the function $\|\cdot\|:L\to[0,\infty)$, $\|\cdot\|:x\mapsto\inf\{r\ge0:x\in rD\}$ is a well-defined norm on the linear space $L$ and the set $D$ coincides with the closed unit ball $B_L$ of the normed space $(L,\|\cdot\|)$. Since the  identity inclusion $T:(L,\|\cdot\|)\to E$ is continuous and the set $D=T(B_L)$ is limited in $E$, the operator $T$ is limited. By (vii), the set $D=T(B_L)$ is barrel-precompact in $E$ and so is the set $B\subseteq D$. Thus $E$ has the Gelfand--Phillips property.
%\smallskip
%(i)$\Rightarrow$(viii) is trivial.
\end{proof}

The following theorem shows that the equivalence (i)$\Leftrightarrow$(ii) in Theorem \ref{t:Drew-GP} remains true for a wider class of locally convex spaces that includes all barrelled normed spaces, in particular, all Banach spaces.

\begin{theorem} \label{t:GP-weak-null}
Let $E$ be a locally convex space such that $E$ has a bounded barrel and $E'$ is dense in $(E_\beta)'_\beta$. Then the space $E$ is Gelfand--Phillips if and only if every limited weakly null sequence in $E$ converges to zero in $E_\beta$.
\end{theorem}

\begin{proof} 
To prove the ``only if'' part, assume that $E$ is Gelfand--Phillips and take any limited weakly null sequence $\{x_n\}_{n\in\w}\subseteq E$. Assuming that the sequence $\{x_n\}_{n\in\w}$ does not converge to zero in the topology $\beta(E,E')$, we can find a barrel $B\subseteq E$ such that the set $I=\{n\in\w:x_n\notin 3B\}$ is infinite. By the Gelfand--Phillips property of $E$, the limited set $\{x_n\}_{n\in\w}$ is barrel-precompact and hence $\{x_n\}_{n\in\w}\subseteq F+B$ for some finite set $F$. By the Pigeonhole Principle, there exists an element $y\in F$ such that the set $J=\{n\in I:x_n\in y+B\}$ is infinite.
Assuming that $y\in 2B$, we would conclude that $x_n\in y+B\subseteq 2B+B=3B$ for every $n\in J\subseteq I$, which contradicts the definition of the set $I$. Therefore, $\frac{1}{2}y\notin B$ and, by the Hahn--Banach Theorem, there exists a functional $\chi\in E'$ such that $\sup \chi(B)<\frac{1}{2}\chi(y)=1$. Then
$$\inf_{n\in J} \chi(x_n)\ge \inf\chi(y+B)\ge \chi(y)-\sup\chi(B)\ge 2-1=1,$$
which means that the sequence $\{x_n\}_{n\in\w}$ is not weakly null in $E$. This contradiction shows that the sequence $\{x_n\}_{n\in\w}$  converges to zero in the topology $\beta(E,E')$.
\smallskip

The proof of the ``if'' part is more complicated. Assume that every limited weakly null sequence in $E$ converges to zero in $E_\beta$. Suppose for a contradiction that the space $E$ is not Gelfand--Phillips. Then $E$ contains a limited set $L$ which is not barrel-precompact, and hence $L$ contains a $B$-separated sequence $\{x_n\}_{n\in\w}$ for some barrel $B$. Since $E$ contains a bounded barrel, we can assume that the barrel $B$ is bounded. Then its gauge is a norm $\|\cdot\|$ in the space $E$ and the sequence $\{x_n\}_{n\in\w}$ is $1$-separated in this norm. Since the limited set $L$ is barrel-bounded, the sequence $\{x_n\}_{n\in\w}$ is bounded in the normed space $(E,\|\cdot\|)$.

\begin{claim}\label{cl:no-Cauchy}  
The sequence $\{x_n\}_{n\in\w}$ contains no subsequences which are weakly Cauchy in the normed space $(E,\|\cdot\|)$.
\end{claim}

\begin{proof} 
To derive a contradiction, assume that the sequence $\{x_n\}_{n\in\w}$ does contain a subsequence $\{x_{n_k}\}_{k\in\w}$ which is weakly Cauchy in the normed space $(E,\|\cdot\|)$. Then the sequence $(x_{n_{k+1}}-x_{n_k})_{k\in\w}$ is weakly null in $(E,\|\cdot\|)$ and hence weakly null in $E$ by the continuity of the identity operator $(E,\|\cdot\|)\to E$. Since the set $\{x_n\}_{n\in\w}\subseteq L$ is limited, so is the set $\{x_{n_{k+1}}-x_{n_k}\}_{k\in\w}$. By our assumption, the weakly null limited sequence $\{x_{n_{k+1}}-x_{n_k}\}_{k\in\w}$ is null in $E_\beta$, which is not possible as $x_{n_{k+1}}-x_{n_k}\notin B$ for every $k\in\w$.
\end{proof}

By Claim~\ref{cl:no-Cauchy} and Rosenthal's $\ell_1$-Theorem \cite[p.201]{Diestel}, $(x_n)_{n\in\w}$ contains a subsequence $(x_{n_k})_{k\in\w}$, equivalent to the standard basis in $\ell_1$.
Replacing $(x_n)_{n\in\w}$ by this subsequence, we can assume that $(x_n)_{n\in\w}$ is equivalent to the standard basis in $\ell_1$. Then the closed linear hull $X$ of  the set $\{x_n\}_{n\in\w}$ in the completion $\hat E$ of the normed space $(E,\|\cdot\|)$ is isomorphic to the Banach space $\ell_1$. Observe that the identity inclusion $I:X\to \hat E$ is not compact. By \cite{BB}, there exists a continuous operator $T:\hat E\to c_0$ such that $T{\restriction}_X=T\circ I:X\to c_0$ is not compact. Let $D$ be the absolutely convex hull of the set $L_0=\{x_n\}_{n\in\w}$. Since $(x_n)_{n\in\w}$ is equivalent to the standard basis in $\ell_1$, the set $D$ is a bounded neighborhood of zero in $X$. 
Then, by the non-compactness of the operator $T{\restriction}_X$, the image $T(D)$ is not precompact in $c_0$. Since absolutely convex hulls of precompact sets are precompact and $T(D)$ is contained in the absolutely convex hull of the set $T(L_0)$ in $c_0$, the set $T(L_0)$ is not precompact in $c_0$ and hence $\|e'_n\|_{T(L_0)}\not\to 0$, where $(e_n')_{n\in\w}$ is the sequence of coordinate functionals in $c_0$. For every $n\in\w$, consider the functional $e'_n\circ T{\restriction}_E$ which is continuous in the norm $\|\cdot\|$ of the space $E$ and hence is a continuous functional on the space $E_\beta$. Since the barrel $B$ is bounded and $E'$ is dense in $(E_\beta)'_\beta$, there exists a functional $\chi_n\in E'$ such that $\|\chi_n-e'_n\circ T\|_{B}<\frac{1}{2^n}$. Then $(\chi_n)_{n\in\w}$ is a weak$^\ast$ null sequence in $E'$ such that $\|\chi_n\|_{L_0}\not\to 0$, which implies that the set $L_0$ is not limited. But this contradicts the choice of the limited set $L\supseteq L_0$.
\end{proof}

We use Theorem \ref{t:GP-characterization} to obtain some hereditary properties of the class of Gelfand--Phillips spaces. We define a linear subspace $X$ of a locally convex space $E$ to be {\em $\beta$-embedded} if the identity inclusion $X_\beta\to E_\beta$ is a topological embedding. It is easy to see that $X$ is $\beta$-embedded in $E$ if and only if for any barrel $B\subseteq X$ there exists a barrel $D\subseteq E$ such that $D\cap X\subseteq B$.
%The following assertion is Proposition 2.12 in \cite{BG-sGP}.

\begin{proposition}\label{p:beta-emb} %[\cite{BG-sGP}]\label{p:beta-emb}
A subspace $X$ of a locally convex space $E$ is $\beta$-embedded if one of the following conditions is satisfied:
\begin{enumerate}
\item[{\rm (i)}] $X$ is complemented in $E$;
\item[{\rm(ii)}] $X$ is barrelled;
\item[{\rm(iii)}] $X$ and $E$ are $\beta$-Banach and $X_\beta$ is closed in $E_\beta$.
\end{enumerate}
\end{proposition}

\begin{proof}
Given a barrel $B\subseteq X$, we should find a barrel $D\subseteq E$ such that $D\cap X\subseteq B$.
\smallskip

(i) If $X$ is complemented in $E$, then there exists a linear continuous operator $R:E\to X$ such that $R(x)=x$ for all $x\in X$. In this case the set $D=R^{-1}(B)$ is a barrel in $E$ with $D\cap X=B$.
\smallskip

(ii) If $X$ is barrelled, then the barrel $B$ is a neighborhood of zero. Since $X$ is a subspace of $E$, there exists a barrel neighborhood $D\subseteq E$ of zero such that $D\cap X\subseteq B$.
\smallskip

(iii) Assume that the spaces $X$ and $E$ are $\beta$-Banach and $X_\beta$ is closed in $E_\beta$. Then the identity inclusion $I:X_\beta\to E_\beta$ is a continuous injective operator between Banach spaces such that the image $I(X_\beta)$ is closed in $E_\beta$. By the Banach Open Mapping Principle, the operator $I:X_\beta\to E_\beta$ is a topological embedding.
\end{proof}

In the next corollary we give some sufficient conditions on a subspace of a  Gelfand--Phillips  space to have the  Gelfand--Phillips property.

\begin{corollary} \label{c:GP-subspace}
Assume that a locally convex space $E$ is Gelfand--Phillips. Then:
\begin{enumerate}
\item[{\rm (i)}] Every $\beta$-embedded subspace of $E$ is Gelfand--Phillips.
\item[{\rm (ii)}] Every barrelled subspace of $E$ is Gelfand--Phillips.
\item[{\rm (iii)}] If $E_\beta$ is barrelled (for example, $E$ is $\beta$-Banach), then $E_\beta$ is Gelfand--Phillips.
\end{enumerate}
\end{corollary}

\begin{proof}
(i) Let $X$ be a $\beta$-embedded subspace of $E$ and $D$ be an infinite barrel-bounded barrel-separated subset of $X$. Take any barrel $B\subseteq X$ such that $D$ is $B$-separated. Since $X$  is $\beta$-embedded, there exists a barrel $B'\subseteq E$ such that $B'\cap X\subseteq B$. Observe that $D$ is also $B'$-separated in $E$. Since $E$ is Gelfand--Phillips, by  Theorem  \ref{t:GP-characterization}, there exist a continuous operator $T:E\to C_p^0(\w)$ such that $T(D)$ is not  precompact in  $c_0$. Then the restriction $T{\restriction}_X$ of $T$ onto $X$ is continuous and $T{\restriction}_X(D)=T(D)$ is not precompact  in  $c_0$. Thus, by Theorem  \ref{t:GP-characterization}, the lcs $X$ is Gelfand--Phillips.
\smallskip

(ii) follows from (i) and  Proposition \ref{p:beta-emb}(ii).
\smallskip

(iii) Assume that the space $E_\beta$ is barrelled and take any infinite barrel-bounded barrel-separated subset $D\subseteq E_\beta$. Find a barrel $B\subseteq E_\beta$ such that $D$ is $B$-separated.  Since $E_\beta$ is barrelled, $B$ is a neighborhood of zero in $E_\beta$. By the definition of $E_\beta$, $B$ is a barrel in $E$. Therefore, by Theorem  \ref{t:GP-characterization}, there exist a continuous operator $T:E\to C_p^0(\w)$ such that $T(D)$ is not  precompact in  $c_0$. It is clear that  for the identity inclusion $i:E_\beta\to E$ and the operator $T\circ i:E_\beta\to C_p^0(\w)$, the image $T\circ i(D)=T(D)$  is not  precompact in  $c_0$. Thus, by Theorem  \ref{t:GP-characterization}, the lcs $E_\beta$ is Gelfand--Phillips.
\end{proof}

%Now we present a sufficient condition for the eJNP in terms of near sequential compactness.

%{\color{blue}
%\begin{lemma}\label{l:nsc-bounded} Let $E$ be a locally convex space. Every near sequentially compact set $S\subset E'_{w^\ast}$ is $\sigma(E',E)$-bounded and its polar $S^\circ$ is a barrel in $E$.
%\end{lemma}

%\begin{proof} Assuming that $S$ is not $\sigma(E',E)$-bounded, we can find a point $x\in E$ such that the set $\{s(x):s\in S\}$ is unbounded in $\IF$. Consequently, for every $n\in\w$ the  set $U_n:=\{s\in S:|s(x)|>n\}$ is non-empty and open in $S$. Since $S$ is near seqeuntially compact, there exists a sequence of points $s_n\in U_n$ containing a  subsequence $(s_{n_k})_{k\in\w}$ that converges to some point $f\in E'_{w^\ast}$. Then $|f(x)|=\lim_{k\to\infty}|s_{n_k}(x)|\ge \lim_{k\to\infty}n_k=\infty$, which is desired contradiction showing that $S$ is $\sigma(E',E)$-bounded. By Theorem 8.3.2 in \cite{Jarchow}, the polar $S^\circ$ is a barrel in $E$.
%\end{proof}
%}

%{t:sgS-eJNP}

Theorem \ref{t:GP-characterization} suggests also to study topological spaces which are selectively sequentially pseudocompact at itself. Following \cite{DAS1}, we call a topological space $X$  {\em selectively sequentially pseudocompact} if $X$ is selectively sequentially pseudocompact at $X$.  Clearly, every sequentially compact space is selectively sequentially pseudocompact, and every selectively sequentially pseudocompact space is pseudocompact.
It is easy to see that a topological space $X$ is selectively sequentially pseudocompact if and only if $X$ is selectively sequentially pseudocompact at some dense set $A\subseteq X$.
Compact selectively sequentially pseudocompact spaces form the class $\KK''$ introduced by Drewnowski and Emmanuele \cite{DrewEm}.

%In particular, each Cantor cube $\{0,1\}^\kappa$ is selectively sequentially pseudocompact because its dense  subset
%\[
%\Sigma=\big\{ x\in \{0,1\}^\kappa:  |\supp(x)|\le\w\big\}
%\]
%is sequentially compact.

\begin{corollary}\label{c:GP-sufficient}
A locally convex space $E$ is Gelfand--Phillips if one of the following conditions holds:
\begin{enumerate}
\item[{\rm (i)}] for every barrel $B\subseteq E$ there exists a barrel $D\subseteq B$ such that $E'_{w^\ast}$ is selectively sequentially precompact at the polar $D^\circ\subseteq E'_{w^\ast}$;
\item[{\rm (ii)}] for every barrel $B\subseteq E$ there exists a barrel $D\subseteq B$ such that $E'_{w^\ast}$ is selectively sequentially pseudocompact at the polar $D^\circ\subseteq E'_{w^\ast}$;
\item[{\rm (iii)}] for any barrel $B\subseteq E$ there exists a barrel $D\subseteq B$ whose polar $D^\circ$ endowed with the weak$^\ast$-topology is selectively sequentially pseudocompact;
\item[{\rm (iv)}] $E$ is separable and barrelled.
\end{enumerate}
\end{corollary}

\begin{proof}
(i) In order to apply Theorem~\ref{t:GP-characterization}(v), fix any infinite barrel-bounded barrel-separated  subset $X$ of $E$. Choose a barrel $B\subseteq E$ such that $X$ is $B$-separated. By (i), there exists a barrel $D\subseteq B$ such that the space $E'_{w^\ast}$ is selectively sequentially precompact at the set $D^\circ$. By the Hahn--Banach Separation Theorem \cite[7.3.5]{Jar}, $\sup_{s\in D^\circ}|s(x)|>1$ for any $x\in E\setminus D$. In particular, $\sup_{s\in D^\circ}|s(x-x')|>1$ for any distinct points $x,x'\in X$. Now it is legal to apply the implication (v)$\Rightarrow$(i) in Theorem~\ref{t:GP-characterization} and conclude that the locally convex space $E$ is Gelfand--Phillips.
\smallskip

(ii) follows from (i) since if $E'_{w^\ast}$ is selectively sequentially pseudocompact at $D^\circ$, then $E'_{w^\ast}$ is also  selectively sequentially precompact at $D^\circ$.
\smallskip

(iii) Observing that for any selectively sequentially pseudocompact subspace $S\subset E'_{w^\ast}$, the space $E'_{w^\ast}$ is selectively sequentially pseudocompact at $S$, we conclude that (iii) follows from (ii).
\smallskip

(iv) Assume that $E$ is barelled and separable. Since $E$ is barrelled, each barrel $B$ in $E$ is a neighborhood of zero. By the Alaoglu--Bourbaki Theorem \cite[8.5.2]{Jar}, the polar $B^\circ$ is compact in $E'_{w^\ast}$. Since $E$ is separable, the compact space $B^\circ$ is metrizable according to \cite[8.5.3]{Jar}. Being metrizable and compact, the space $B^\circ$ is selectively sequentially pseudocompact. Applying (iii), we conclude that $E$ has the Gelfand--Phillips property.
\end{proof}

In the next proposition we use the following simple lemma.
\begin{lemma}\label{l:C0-c0-cont}
Each continuous operator $T:L\to C_p^0(\w)$ from a barrelled space $L$ remains continuous as an operator from $L$ to $c_0$.
\end{lemma}

\begin{proof}
Let $\{e'_n\}_{n\in\w}$ be the sequence of coordinate functionals on the Banach space $c_0$. The definition of the topology of the space $C_p^0(\w)$ ensures that each functional $e_n'$ remains continuous on the locally convex space $C_p^0(\w)$. Observe that the intersection $B:=\bigcap_{n\in\w}\{x\in C_p^0(\w):|e'_n(x)|\le 1\}$ coincides with the closed unit ball of the Banach space $c_0$. Since $B$ is a barrel also in $C^0_p(\w)$, the continuity of the operator $T$ implies that the set $T^{-1}(B)$ is a barrel in $L$. Since $L$ is barrelled, $T^{-1}(B)$ is a neighborhood of zero, which means that the operator $T:L\to c_0$ is continuous.
\end{proof}

Let $Y$ be a locally convex space. A locally convex space $E$ is defined to have the {\em separable $Y$-extension property} if every separable subspace of $E$ is contained in a barelled separable linear subspace $X\subseteq E$ such that each  continuous operator $T:X\to Y$ can be extended to a continuous operator $\bar T:E\to Y$. This definition implies that each separable barrelled locally convex space has the separable $Y$-extension property for any locally convex space $Y$.

\begin{proposition} \label{p:barrelled-c0-extension}
A \textup{(}barrelled\textup{)} locally convex space $E$ has the separable $C_p^0(\w)$-extension property if \textup{(}and only if\/\textup{)} $E$ has the separable $c_0$-extension property.
\end{proposition}

\begin{proof}
Assume that  $E$ has the separable $c_0$-extension property, and let $H$ be a separable subspace of $E$. By the separable $c_0$-extension property, $H$ is contained in a barrelled separable linear subspace $L\subseteq E$ such that each continuous operator $T:L\to c_0$ can be extended to a continuous operator $\bar T:E\to c_0$. Let $T:L\to C_p^0(\w)$ be a continuous operator. Since $L$ is barrelled,  Lemma \ref{l:C0-c0-cont} implies that $T$ is also continuous as an operator from $L$ to the Banach space $c_0$. Therefore $T$ can be extended to a continuous operator $\bar T:E\to c_0$. Clearly, $\bar T$ is also continuous as an operator from $E$ to $C_p^0(\w)$. Thus $E$ has the separable $C_p^0(\w)$-extension property.

If $E$ is barrelled and has the separable $C_p^0(\w)$-extension property, then an analogous argument shows that $E$ has the separable $c_0$-extension property.
%Assume now that $E$ is barrelled and has the separable $C_p^0(\w)$-extension property. Fix a separable subspace $H$ of $E$ and find a barrelled subspace $L$ of $E$ containing $H$ such that every continuous linear operator $T:L\to C_p^0(\w)$ can be extended to a continuous linear operator $\bar T:E\to C_p^0(\w)$. Let $T:L\to c_0$ be a continuous linear operator. Clearly, $T$ is continuous as an operator from  $L$ to $C_p^0(\w)$ and hence it can be extended to a continuous linear operator $\bar T:E\to C_p^0(\w)$. Now Claim 1 applies.
\end{proof}

%Proposition \ref{p:barrelled-c0-extension} can be easily derived from the following lemma.

%\begin{lemma}
%Each linear continuous operator $T:E\to C_p^0(\w)$ from a barrelled space $E$ remains continuous as an operator from $E$ to $c_0$.
%\end{lemma}

%\begin{proof}
%Let $\{e'_n\}_{n\in\w}$ be the sequence of coordinate functionals on the Banach space $c_0$. The definition of the topology of the space $C_p^0(\w)$ ensures that each functional $e_n'$ remains continuous on the locally convex space $C_p^0(\w)$. Observe that the intersection $B:=\bigcap_{n\in\w}\{x\in C_p^0(\w):|e'_n(x)|\le 1\}$ coincides with the closed unit ball of the Banach space $c_0$. The continuity of the operator $T$ implies that the set $D=\{x\in E:\forall n\in\w\;|e_n'\circ T(x)|\le1\}=T^{-1}(B)$ is a barrel in $E$. Since $E$ is barrelled, $T^{-1}(B)$ is a neighborhood of zero, which means that the operator $T:E\to c_0$ is continuous.
%\end{proof}

Using the classical Sobczyk Theorem \cite[p.72]{Diestel} (which states that if $H$ is a linear subspace of a separable Banach space $E$ and $T:H\to c_0$ is a bounded  operator, then there is  a bounded operator $S:E\to c_0$ extending $T$ to the whole $E$), one can prove the following characterization showing that for Banach spaces our definition of separable $c_0$-extension property is equivalent to that introduced by Correa and Tausk \cite{CT13,CT14}.

\begin{proposition}
A Banach space $E$ has the  separable $c_0$-extension property if and only if every continuous operator $T:X\to c_0$ defined on a separable subspace $X\subseteq E$ can be extended to a  continuous operator $\bar T:E\to c_0$.
\end{proposition}

By \cite{CT13,CT14} the class of Banach spaces with the separable $c_0$-extension property includes all weakly compactly generated Banach spaces, all Banach spaces with {\em the separable complementation property} (=every separable subspace is contained in a separable complemented subspace), and all Banach spaces $C(K)$ over $\aleph_0$-monolithic compact lines $K$. Let us recall that a topological space $X$ is {\em $\aleph_0$-monolithic} if each separable subspace of $X$ has a countable network.

\begin{theorem}\label{t:GP-extension}
Every locally convex space $E$ with the separable $C_p^0(\w)$-extension property is Gelfand--Phillips.
\end{theorem}

\begin{proof}
In order to apply Theorem~\ref{t:GP-characterization}(iv), fix  an infinite barrel-bounded, barrel-separated set $D\subseteq E$. Choose any countable infinite subset $I\subseteq D$. By the separable $C_p^0(\w)$-extension property, $I$ is contained in a barrelled separable linear subspace $X\subseteq E$ such that every  continuous operator $T:X\to C_p^0(\w)$ can be extended to a  continuous operator $\bar T:E\to C_p^0(\w)$.

By Corollary~\ref{c:GP-sufficient}(iv), the barrelled separable lcs $X$ is Gelfand--Phillips. By Theorem~\ref{t:GP-characterization}, there exists a continuous operator $T:X\to C_p^0(\w)$ such that the set $T(I)$ is not precompact in the Banach space $c_0$. By the choice of $X$, the operator $T$ can be extended to a continuous operator $\bar T:E\to C_p^0(\w)$. Observing that the set $\bar T(D)\supseteq T(I)$ is not precompact in $c_0$, we can apply Theorem~\ref{t:GP-characterization} and conclude that $E$ is Gelfand--Phillips.
\end{proof}

Proposition \ref{p:barrelled-c0-extension} and Theorem \ref{t:GP-extension} imply
\begin{corollary} \label{c:barrelled-exten-GP}
A barrelled space with the separable $c_0$-extension property is Gelfand--Phillips.
\end{corollary}

%\begin{remark}
%{\rm Below (see Example~\ref{ex:split}) we construct a  Banach space $E$ with the eJNP that does not have the separable $c_0$-extension property.\qed}
%\end{remark}

We finish this section with some conditions on a locally convex space implying the failure of the Gelfand--Phillips property.

%A continuous operator $T:X\to Y$ between locally convex spaces is defind to be {\em $\beta$-to-$\beta$ compact} if for any barrel-bounded set $B\subseteq X$ the image $T(B)$ is barrel-precompact in $Y$. Observe that a continuous  operator $T:X\to Y$ between Banach spaces is $\beta$-to-$\beta$ compact if and only if it is compact in the standard sense.

\begin{theorem}\label{t:not-eJNP}
A locally convex space $E$ is not Gelfand--Phillips if  the identity map $E'_{w^\ast}\to (E'_\beta)_w$ is sequentially continuous and $E$ admits a  continuous operator $T:c_0\to E$ which is not $\beta$-to-$\beta$ precompact.
\end{theorem}

\begin{proof}
Since the operator $T:c_0\to E$ is not $\beta$-to-$\beta$ precompact, the image $T(B)$ of the closed unit ball $B=\{x\in c_0:\|x\|_0\le 1\}$ of $c_0$ is not barrel-precompact in $E$. Assuming that the space $E$ is Gelfand--Phillips, we can find a  null sequence $\{\mu_n\}_{n\in\w}\subseteq E'_{w^\ast}$ such that $\|\mu_n\|_{T(B)}\not\to 0$. By our assumption, the identity map $E'_{w^\ast}\to (E'_\beta)_w$ is sequentially continuous, which implies that the sequence $\{\mu_n\}_{n\in\w}$ converges to zero in the weak topology of the strong dual space $E'_\beta$. Then for the dual operator $T^\ast:(E_\beta)'_w\to (c_0)'_w=(\ell_1)_w$, the sequence $\{T^\ast\mu_n\}_{n\in\w}$ converges to zero in the weak topology of the Banach space $\ell_1$. By the Schur Theorem \cite[VII]{Diestel}, this sequence converges to zero in norm. Observe that for every $n\in\w$ and each $x\in B$, we have
\[
\|\mu_n\|_{T(B)}=\sup_{x\in B}|\mu_n(T(x))|=\sup_{x\in B}|(T^\ast\mu_n)(x)|=\|T^\ast\mu_n\|\to 0,
\]
which contradicts the choice of the sequence $\{\mu_n\}_{n\in\w}$.
\end{proof}

%%%%%%%%%%%%%%%%%%%%%%%%%%%%%%%%%%%%%%%%%%%%%%%
%%%%%%%%%%%%%%%%%%%%%%%%%%%%%%%%%%%%%%%%%%%%%%%
%%%%%%%%%%%%%%%%%%%%%%%%%%%%%%%%%%%%%%%%%%%%%%%
%%%%%%%%%%%%%%%%%%%%%%%%%%%%%%%%%%%%%%%%%%%%%%%

\section{The Gelfand--Phillips property in function spaces} \label{sec:GP-f}

%%%%%%%%%%%%%%%%%%%%%%%%%%%%%%%%%%%%%%%%%%%%%%%
%%%%%%%%%%%%%%%%%%%%%%%%%%%%%%%%%%%%%%%%%%%%%%%
%%%%%%%%%%%%%%%%%%%%%%%%%%%%%%%%%%%%%%%%%%%%%%%
%%%%%%%%%%%%%%%%%%%%%%%%%%%%%%%%%%%%%%%%%%%%%%%

In this section we apply the results from the preceding sections to study the  Gelfand--Phillips property in function spaces. %   $C_p(X)$ and $\CC(X)$.

We shall use repeatedly the following assertion whose proof can be found in \cite{BG-JNP}.

\begin{proposition}\label{p:bb-Ck}
Let $X$ be a Tychonoff space, and let $\TTT$ be a locally convex topology on $C(X)$ such that $\TTT_p\subseteq \TTT\subseteq \TTT_k$. Then:
\begin{enumerate}
\item[{\rm(i)}] for every barrel $D$ in $C_\TTT(X)$, there are a functionally bounded subset $A$ of $X$ and $\e>0$ such that $[A;\e]\subseteq D$.
\item[{\rm(ii)}] a subset $\F\subseteq C_\TTT(X)$  is barrel-bounded if and only if for any functionally bounded set $A\subseteq X$, the set $\F(A):=\bigcup_{f\in\F}f(A)$ is bounded in $\IF$;
\item[{\rm(iii)}] $\big(C_\TTT(X)\big)_{\!\beta} = C_b(X)$.
\end{enumerate}
\end{proposition}

Below we give examples of locally convex spaces without the $(GP)$ property. Recall that a Tychonoff space $X$ is called an {\em $F$-space} if every functionally open set $A$ in $X$ is {\em $C^\ast$-embedded} in the sense that every bounded continuous function $f:A\to \IR$ has a continuous extension $\bar f:X\to\IR$. For  numerous equivalent conditions for a Tychonoff space $X$ being an $F$-space, see \cite[14.25]{GiJ}. In particular, the Stone--\v{C}ech compactification $\beta \Gamma$ of a discrete space $\Gamma$ is a compact $F$-space. %The following example complements Example \ref{exa:F-space-JNP}. % and gives an example of Banach spaces $C(K)$ without the $(GP)$ property.

\begin{example}\label{exa:F-space-eJNP}
For any infinite compact $F$-space $K$, the spaces $C(K)$ and  $C_p(K)$ are not Gelfand--Phillips.
\end{example}

\begin{proof}
We proved in \cite{BG-GP-Banach} that for any compact $F$-space $K$ the Banach space $C(K)$ is not Gelfand--Phillips. By Proposition~\ref{p:bb-Ck}, $(C_p(K))_\beta=C(K)$. By (iii) of Corollary~\ref{c:GP-subspace}, the function space $C_p(K)$ is not Gelfand--Phillips.
\end{proof}

Below we provide a sufficient condition on the space $X$ for which $C_p(X)$ is Gelfand--Phillips.

\begin{theorem}\label{t:Cp-GP}
Let $X$ be a Tychonoff space such that every functionally bounded subset $K\subseteq X$ is contained in a subset $K'$ such that $X$ is  selectively sequentially pseudocompact at $K'$. Then the  function space $C_p(X)$ is Gelfand--Phillips.
\end{theorem}

\begin{proof}
In order to apply Theorem~\ref{t:GP-characterization} to the locally convex space $E=C_p(X)$, take any infinite barrel-bounded barrel-separated subset $A$ of $E$. Find a barrel $B\subseteq E$ such that $A$ is $B$-separated. Choose an arbitrary sequence of pairwise distinct elements  $\{a_n\}_{n\in\w}$ of $A$. Then the set $\{ a_n-a_m: n,m\in\w, \, n\not= m\}$ is contained in $E\SM B$.
By (i) of Proposition \ref{p:bb-Ck}, there exists a functionally bounded closed set $K\subseteq X$ and a real number $\e>0$ such that the barrel $D:=\{f\in C_p(X):\|f\|_K\le\e\}$ is contained in $B$. By our assumption, $K$ is contained in a subset  $K'\subseteq X$ such that $X$ is  selectively sequentially pseudocompact at $K'$. Then for any element $a_n-a_m\in  E\setminus B\subseteq E\setminus D$ we have
\[
\sup_{x\in K'}|a_n(x)-a_m(x)|\ge\sup_{x\in K}|a_n(x)-a_m(x)|>\e.
\]
Let $\delta:X\to E'_{w^\ast}$ be the continuous function assigning to each point $x\in X$ the evaluation functional $\delta_x\in E'$ at $x$ (so $\delta_x(f)=f(x)$ for every $f\in C_p(X)$). Consider the set $S:=\{\frac{1}{\e}\delta_x:x\in K'\}\subseteq E'_{w^\ast}$. It is well known that $\delta$ is a homeomorphic embedding with the closed image. This fact and the selective sequential pseudocompactness of $X$ at $K'$ imply that the space $E'_{w^\ast}$ is  selectively sequentially pseudocompact at $\delta(K')$ and also at its homothetic copy $S=\frac{1}{\e}\delta(K')$. Observe that for every  positive integers $n\ne m$, we have
\[
\sup_{s\in S}|s(a_n-a_m)|=\tfrac{1}{\e}\sup_{x\in K'}|a_n(x)-a_m(x)|>\tfrac{1}{\e} \cdot\e=1.
\]
Applying Theorem~\ref{t:GP-characterization}, we conclude that the locally convex space $E=C_p(X)$ is Gelfand--Phillips.
\end{proof}

Corollary~\ref{c:GP-subspace} and Theorem~\ref{t:Cp-GP} imply

\begin{corollary}\label{c:GP-sspC}
For every  selectively sequentially pseudocompact space $K$, the function spaces $C_p(K)$ and $C(K)$ are Gelfand--Phillips.
\end{corollary}

Since the spaces $[0,\w_1]$ and $[0,\w_1)$ are sequentially compact, %Theorem~\ref{t:BG-strong-GP} and
Corollary~\ref{c:GP-sspC} implies

\begin{example}
The spaces $C_p[0,\w_1]$ and $C_p[0,\w_1)$ are Gelfand--Phillips.% but they do not have the strong $(GP)$ property.
\end{example}

In the next two theorems we show that  the Gelfand--Phillips property for function spaces satisfies some ``hereditary property'' in the sense that if $C(X)$ has the Gelfand--Phillips property for a topology $\TTT$ then the space $C(X)$ has this property also for a finer topology. For a locally convex topology $\TTT$ on $C(X)$ we will denote the locally convex space $(C(X),\Tau)$ by $C_\TTT(X)$.

\begin{theorem} \label{t:GP-between}
Let $X$ be a Tychonoff space, and let $\tau$ and $\TTT$ be locally convex topologies on $C(X)$ such that $\TTT_p\subseteq\tau\subseteq\TTT\subseteq\TTT_k$. If the space $C_{\tau}(X)$ is Gelfand--Phillips, then so is $C_{\TTT}(X)$. In particular, if $C_p(X)$ is Gelfand--Phillips, then so is $\CC(X)$.
\end{theorem}

\begin{proof}
%For simplicity of notations we set $Z:=C_{\TTT_2}(X)$ and $Y:=C_{\TTT_1}(X)$.
Fix any barrel-bounded set $B\subseteq C_{\TTT}(X)$, which is not barrel-precompact in $C_{\TTT}(X)$. The continuity of the identity operator $C_{\TTT}(X)\to C_{\tau}(X)$ implies that $B$ is barrel-bounded in $C_{\tau}(X)$.

We claim that $B$ is not barrel-precompact in $C_\tau(X)$. Indeed, since $B$  is not barrel-precompact in $C_\TTT(X)$, there exists a barrel $D$ in $C_\TTT(X)$ such that $B\nsubseteq F+D$ for every finite set $F\subseteq B$. By (i) of Proposition \ref{p:bb-Ck}, there are a functionally bounded subset $A$ of $X$ and $\e>0$ such that $[A;\e]\subseteq D$. It is clear that $[A;\e]$ is a barrel in $C_p(X)$ and hence also in $C_\tau(X)$. Since $[A;\e]\subseteq D$ it follows that $B\nsubseteq F+[A;\e]$ for every  finite set $F\subseteq B$, which means that $B$ is  not barrel-precompact in $C_\tau(X)$.

Since the lcs $C_\tau(X)$ is Gelfand--Phillips, there exists a weak$^\ast$ null sequence $\{\mu_n\}_{n\in\w}$ in $C_\tau(X)'$ such that $\|\mu_n\|_B\not\to 0$. The continuity of the identity operator $C_\TTT(X)\to C_\tau(X)$ ensures that the sequence $\{\mu_n\}_{n\in\w}$ remains weak$^\ast$ null in $C_\TTT(X)'$ and hence $C_\TTT(X)$ is Gelfand--Phillips.
\end{proof}

Theorems \ref{t:Cp-GP} and \ref{t:GP-between} immediately imply the next assertion.
\begin{corollary}\label{c:mu-Cp-Ck-GP}
Let $X$ be a $\mu$-space whose compact subsets are selectively sequentially pseudocompact $($for example $X$ is metrizable$)$. Then the function spaces $C_p(X)$ and $\CC(X)$ are Gelfand--Phillips.
\end{corollary}

%{\color{red}By a {\em compactification} of a Tychonoff space $X$ we understand any compact Hausdorff space $\gamma X$ containing $X$ as a dense subspace. Any compactification $\gamma X$ of $X$ determines the family $\fB_\gamma$ of all subsets $B\subseteq X$ such that for every function $f\in C(X)$ there exists a function $g\in C(\gamma X)$ such that $g{\restriction}_B=f{\restriction}_B$. It is clear that every set $B\in\fB_\gamma$ is functionally bounded in $X$. The Tietze-Urysohn Theorem \cite[3.11]{GiJ} ensures that the family $\fB_\gamma$ contains all compact subsets of $X$. For the Stone-\v{C}ech compactification $\beta X$ of $X$ the family $\fB_\beta$ coincides with the family of all functionally bounded sets in $X$. We shall denote by $\TTT_\gamma$ the locally convex vector topology on $C(X)$ defined by the family $\fB_\gamma$  and set $C_{\TTT_\gamma}(X):=(C(X),\TTT_\gamma)$. It is clear that $\TTT_k \subseteq \TTT_\gamma \subseteq \TTT_b$.}\footnote{\color{red}Do we indeed need those exotic topologies $\Tau_\gamma$? What for?}

Under an additional assumption that the space $C_b(X)$ is barrelled we can strengthen Theorem \ref{t:GP-between} as follows.

\begin{theorem} \label{t:GP-between-b}
Let $X$ be a Tychonoff space such that the space $C_b(X)$ is barrelled, and let $\tau$ and $\TTT$ be locally convex topologies on $C(X)$ such that $\TTT_p\subseteq\tau\subseteq\TTT\subseteq\TTT_b$. If the space $C_{\tau}(X)$ is Gelfand--Phillips, then so is $C_{\TTT}(X)$; in particular, $C_b(X)$ is Gelfand--Phillips. 
\end{theorem}

\begin{proof}
We give a detailed proof although it is similar to the proof of Theorem \ref{t:GP-between}.
%For simplicity of notations we set $Z:=C_{\TTT_2}(X)$ and $Y:=C_{\TTT_1}(X)$.
Assume that $C_\tau(X)$ is Gelfand--Phillips.  To show that $C_\TTT(X)$ is Gelfand--Phillips, take any barrel-bounded set $B\subseteq C_\TTT(X)$ which is not barrel-precompact. The continuity of the identity operator $C_\TTT(X)\to C_\tau(X)$ implies that $B$ is barrel-bounded in $C_\tau(X)$.

We claim that $B$ is not barrel-precompact in $C_\tau(X)$. Indeed, since $B$  is not barrel-precompact in $C_\TTT(X)$, there exists a barrel $D$ in $C_\TTT(X)$ such that $B\nsubseteq F+D$ for every finite set $F\subseteq B$. The continuity of the operator $C_b(X)\to C_\TTT(X)$ implies that $B$ is not barrel-precompact in $C_b(X)$ and $D$ is a barrel in $C_b(X)$. Since $C_b(X)$ is barrelled, $D$ is a neighborhood of zero and hence there are a functionally bounded subset $A$ of $X$ and $\e>0$ such that $[A;\e]\subseteq D$.  It is clear that $[A;\e]$ is a barrel in $C_p(X)$ and hence also in $C_\tau(X)$. Since $[A;\e]\subseteq D$ it follows that $B\nsubseteq F+[A;\e]$ for every  finite set $F\subseteq B$, which means that $B$ is  not barrel-precompact in $C_\tau(X)$.

Since $C_\tau(X)$ is Gelfand--Phillips, there exists a weak$^\ast$ null sequence $\{\mu_n\}_{n\in\w}$ in $C_\tau(X)'$ such that $\|\mu_n\|_B\not\to 0$. By the continuity of  the identity operator $C_\TTT(X)\to C_\tau(X)$, the sequence $\{\mu_n\}_{n\in\w}$ remains weak$^\ast$ null in $C_\TTT(X)'$ and hence $C_\TTT(X)$ is Gelfand--Phillips.
%
%{\color{red}To prove the last assertion it suffices to put $\tau:=\TTT_k$ and $\TTT:=\TTT_\gamma$ and recall that $\TTT_k \subseteq  \TTT_\gamma \subseteq \TTT_b$.}
\end{proof}

We denote by $\upsilon X$, $\mu X$ and $\beta X$ the Hewitt completion (=realcompactification), the Diedonn\'e completion and  the Stone-\v Cech compactification of $X$, respectively. It is known (\cite[8.5.8]{Eng}) that $X\subseteq \mu X\subseteq \upsilon X \subseteq \beta X$. Also it is known that all paracompact spaces and all realcompact spaces are Diedonn\'e complete and each Diedonn\'e complete space is a $\mu$-space, see \cite[8.5.13]{Eng}. On the other hand, each pseudocompact $\mu$-space is compact.

Although Theorem \ref{t:GP-between-b} is quite general, an answer to the following question is open.

\begin{problem}
Let $X$ be a Tychonoff space such that the space $\CC(X)$ is Gelfand--Phillips. Are the spaces $\CC(\upsilon X)$, $\CC(\mu X)$ or $C_b(X)$ Gelfand--Phillips?
\end{problem}

The aforementioned results suggest the following intriguing problem.

\begin{problem}
{\rm(i)} Is there a compact space $K$ such that the Banach space $C(K)$ is Gelfand--Phillips but the function space $C_p(K)$ is not Gelfand--Phillips?

{\rm(ii)} Is there a pseudocompact space $K$  such that the Banach space $C(K)$ is Gelfand--Phillips but the function space $\CC(K)$ is not Gelfand--Phillips?

{\rm(iii)} Is there a Tychonoff space $X$  such that $\CC(X)$ is Gelfand--Phillips but the function space $C_p(K)$ is not Gelfand--Phillips?
\end{problem}

\begin{proposition}\label{p:GP-image}
Let $f:L\to K$ be a surjective map between two pseudocompact spaces. If the function space $C_p(L)$ \textup{(}resp. $C(L)$\textup{)} is Gelfand--Phillips, then also the space $C_p(K)$ \textup{(}resp. $C(K)$\textup{)} is Gelfand--Phillips.
\end{proposition}

\begin{proof}
It follows that the dual operators $T_p:C_p(K)\to C_p(L)$ and $T:C(K)\to C(L)$ assigning to each continuous function $\varphi:K\to \IF$ the composition $\varphi\circ f:L\to\IF$ are isomorphic topological embeddings. The Banach subspace $T(C(K))$ of $C(L)$ is $\beta$-embedded into $C(L)$ by Proposition~\ref{p:beta-emb}(ii). If the Banach space $C(L)$ is Gelfand--Phillips, then, by  Corollary~\ref{c:GP-subspace}(i),  the space $T(C(K))$ is Gelfand--Phillips and so is its isomorphic copy $C(K)$.

Now assume that the function space $C_p(L)$ is Gelfand--Phillips.  By (iii) of Proposition \ref{p:bb-Ck}, $(C_p(K))_\beta=C(K)$ and $(C_p(L))_\beta=C(L)$, which implies that the subspace $T_p(C_p(K))$ is $\beta$-embedded into $C_p(L)$. By Corollary~\ref{c:GP-subspace}, the space $T_p(C_p(K))$ is Gelfand--Phillips and so is its isomorphic copy $C_p(K)$.
\end{proof}

In fact, the selective sequential pseudocompactness of the compact space $K$ in Corollary~\ref{c:GP-sspC} can be replaced by the selective sequential  precompactness of the space $P_\w(K)$ of finitely supported probability measures on $K$ in $C(K)'_{w^\ast}$.

For a Tychonoff space $X$ by a {\em finitely supported probability measure} on $X$ we understand an element of the convex hull $P_\w(X)$ of the set $\{\delta_x:x\in X\}$ of Dirac measures in the dual space $C_p(X)'_{w^\ast}$. The {\em Dirac measure} $\delta_x$ concentrated at a point $x\in X$ is the linear continuous functional $\delta_x:C_p(X)\to\IF$, $\delta_x:f\mapsto f(x)$. For a subset $K\subseteq X$, let $P_\w(K)$ be the convex hull of the set $\{\delta_x:x\in K\}\subseteq C_p(X)'_{w^\ast}$ of Dirac measures concentrated at points of the set $K$.

\begin{proposition}\label{p:GP-P(K)w}
For a Tychonoff space $X$, the locally convex space $E=C_p(X)$ is Gelfand--Phillips whenever for every closed functionally bounded set $K$ in $X$ the space $E'_{w^\ast}$ is selectively sequentially precompact at $P_\w(K)$.
\end{proposition}

\begin{proof}
In order to apply Theorem~\ref{t:GP-characterization} to the locally convex space $E=C_p(X)$, take any infinite barrel-bounded, barrel-separated subset $A$ of $E$. Find a barrel $B\subseteq E$ such that $A$ is $B$-separated. By (i) of Proposition \ref{p:bb-Ck}, there exist a closed functionally bounded set  $K\subseteq X$ and a real number $\e>0$ such that $\{f\in C_p(X):\|f\|_K\le \e\}\subseteq B$. Assuming that $E'_{w^\ast}$ is  selectively sequentially precompact at the set $P_\w(K)$, we conclude that $E'_{w^\ast}$ is also selectively sequentially precompact at the set $S=\frac{1}{\e} P_\w(K)$. Since $A$ is $B$-separated, it follows that for any distinct points $a,b\in A$, we have $a-b\notin B$ and hence $\|a-b\|_K>\e$, which implies
\[
\sup_{s\in S}|s(a-b)|=\tfrac{1}{\e}\sup_{\mu\in P_\w(K)}|\mu(a-b)|\ge\tfrac{1}{\e}\sup_{x\in K}|a(x)-b(x)|=\tfrac{1}{\e}\|a-b\|_K>1.
\]
By Theorem~\ref{t:GP-characterization}, the locally convex space $C_p(X)$ is Gelfand--Phillips.
%
%For every subspace $K$ of $X$, the set $P_\w(K)$ is a subspace of $P_\w(X)$. Now the last assertion immediately follows from the trivial fact that a subset of a near sequentially precompact subset of a topological vector space is also near sequentially precompact.
\end{proof}

%\begin{example}\label{exa:XY}
%There are compact Hausdorff spaces $X\subseteq Y$ such that \\
%{\em(i)} the function space $C_p(Y)$ has the $(GP)$ property but $C_p(X)$ does not have the JNP, and \\
%{\em(ii)} the Banach space $C(Y)$ has the $(GP)$ property but the Banach space $C(X)$ does not have the $(GP)$ property.
%\end{example}

%\begin{proof}
%In the Cantor cube $Y:=\{0,1\}^{\mathfrak c}$ of weight $\mathfrak c$ take a subspace $X$, homeomorphic to $\beta\w$. Since the Cantor cube $\{0,1\}^{\mathfrak c}$ is selectively sequentially pseudocompact, by Corollary~\ref{c:sspC-eJNP}, the function spaces $C_p(Y)$ and $C(Y)$ have the $(GP)$ property. By Example~\ref{exa:F-space-eJNP}, the Banach space $C(X)$ does not have the $(GP)$ property and, by Example~\ref{exa:F-space-JNP}, the function space $C_p(X)$ does not have the JNP.
%\end{proof}

%\begin{remark}{\rm
%Let $X\subseteq Y$ be the compact Hausdorff spaces from Example~\ref{exa:XY}. Since the Banach space $C(X)$ is a quotient of the Banach space $C(Y)$, the $(GP)$ property is not preserved by quotients of Banach spaces. Since the function space $C_p(X)$ is a quotient of the function space $C_p(Y)$, the JNP is not preserved by quotients of locally convex spaces. This contrasts with a result from \cite{BG-GP-Banach} saying that the strong $(GP)$ property is preserved by Banach space quotients.\qed}
%\end{remark}

For a compact space $K$, the space
\[
P(K):=\{\mu\in C(X)':\|\mu\|=\mu(\mathbf{1}_X)=1\}
\]
endowed with the weak$^\ast$ topology, is called the {\em space of probability measures} on $K$. By the Riesz Representation Theorem \cite[7.6.1]{Jar}, each functional $\mu\in P(K)$ can be identified with a regular probability Borel measure on $K$. Each point $x\in K$ can be identified with the Dirac measure $\delta_x\in C(K)'$, $\delta_x:f\mapsto f(x)$. It is well known that the map $\delta:K\to P(K)$, $\delta:x\mapsto \delta_x$, is a topological embedding and the convex hull $P_\w(K)$ of the set $\delta(K)\subseteq P(K)$ is dense in $P(K)$. %Elements of $P_\w(K)$ are called {\em finitely supported probability measures} on $K$.

\begin{corollary} \label{c:GP-P(K)}
Let $K$ be an infinite compact space. Then:
\begin{enumerate}
\item[{\rm (i)}]  if $P(K)$ is selectively sequentially pseudocompact, then $C_p(P(K))$, $C(P(K))$ and $C(K)$ are Gelfand--Phillips;
\item[{\rm (ii)}] if $P(K)$ is selectively sequentially pseudocompact at $P_\w(K)$, then the spaces $C_p(K)$ and $C(K)$ are Gelfand--Phillips.
\end{enumerate}
\end{corollary}

\begin{proof}
(i) By Corollary \ref{c:GP-sspC}, the spaces $C_p(P(K))$ and  $C(P(K))$ are Gelfand--Phillips. The space $C(K)$  is Gelfand--Phillips by Corollary 2.2(vi) of \cite{BG-GP-Banach}.

(ii)  If $P(K)$ is selectively sequentially pseudocompact at $P_\w(K)$, then also $E'_{w^\ast}$ is selectively sequentially precompact at $P_\w(K)$. Therefore, by Proposition \ref{p:GP-P(K)w}, the space $C_p(K)$  Gelfand--Philipps and so is the Banach space $C(K)=C_k(K)$, see Theorem \ref{t:GP-between}.
\end{proof}

The path-connected space $P(K)$ can fail to be selectively sequentially pseudocompact as the following example shows.

\begin{example}
Let $K$ be an infinite compact $F$-space. Then the function spaces $C_p(P(K))$ and $C(P(K))$ are not Gelfand--Phillips. Consequently, the compact space $P(K)$ is not selectively sequentially pseudocompact.
\end{example}

\begin{proof}
Observe that the operator $T:C(K)\to C(P(K))$ assigning to each function $f\in C(K)$ the function $Tf:P(K)\to \IF$, $Tf:\mu\mapsto \mu(f)$, is an isomorphic embedding of the Banach space $C(K)$ into the Banach space $C(P(K))$. Assuming that the Banach space $C(P(K))$ is Gelfand--Phillips, we can apply Corollary~\ref{c:GP-subspace}(ii) and conclude that also the Banach space $C(K)$ is Gelfand--Phillips. But this contradicts Example~\ref{exa:F-space-eJNP}. This contradiction shows that the Banach space $C(P(K))$ is not Gelfand--Phillips.

Now Theorem \ref{t:GP-between} implies the function space $C_p(P(K))$ is not Gelfand--Phillips, and hence, by Corollary~\ref{c:GP-sspC}, the compact space $P(K)$ is not selectively sequentially pseudocompact.
\end{proof}

By Corollary \ref{c:GP-sspC}, if a compact space $K$ is selectively sequentially pseudocompact, then the space $C_p(K)$ is Gelfand--Phillips. So it is natural to ask whether the converse is true, namely, let $K$ be a compact space such that the space $C_p(K)$ is Gelfand--Phillips. {\em Is then $K$ selectively sequentially pseudocompact$?$} Since every selectively sequentially pseudocompact space contains non-trivial convergent sequences, one can also ask a much weaker question: {\em Does $K$ contain a non-trivial convergent sequence}? Under an additional Set-Theoretic assumption (namely, the Jensen Diamond Principle $\diamondsuit$) we answer this question in the negative.
Recall that a compact space $X$ is called {\em Efimov} if $X$ contains neither a non-trivial convergent sequence nor topological copy of $\beta\w$.

%In \cite{DP} and \cite{DM}, a simple Efimov space is constructed assuming  the Jensen Diamond Principle $\diamondsuit$ and Continuum Hypothesis, respectively.
%{\color{blue}Grzegorz Plebanek suggested to the authors that possibly the example of a simple Efimov space constructed in Theorem 4.2 of \cite{DP} answers our aforementioned question in the negative. Below we realize Plebanek's conjecture.

%Simple compact Hausdorff spaces have many nice properties. In particular:
%\begin{itemize}
%\item[(i)] a simple compact Hausdorff space $X$ does not admit a continuous map onto $[0,1]^{\w_1}$ and hence $X$ contains no topological copies of $\beta\w$;
%\item[(ii)] every nonatomic measure $\mu\in P(X)$ on a simple compact Hausdorff space $X$ is {\em uniformly regular} in the sense that there exists a continuous map $f:X\to M$ to a compact metrizable space $M$ such that $\mu(F)=\mu(f^{-1}(f(F))$ for any closed subset $F\subset X$.
%\end{itemize}

\begin{example}\label{t:diamond}
%{\rm(i)}
Under $\diamondsuit$, there exists a Efimov  space $X$ whose function space $C_p(X)$ is Gelfand--Phillips.
%{\rm(ii)} Under Continuum Hypothesis, there exists a compact spaces $X$ without non-trivial convergent sequences such that  the function space $C_p(X)$ has the $(GP)$ property.
\end{example}

\begin{proof}
Under $\diamondsuit$,  we constructed in \cite{BG-Efimov} a simple Efimov  space $X$ such that $P(X)$ is selectively sequentially pseudocompact. By Corollary~\ref{c:GP-P(K)}, the function space $C_p(X)$ is Gelfand--Phillips.
%
%(ii) Under CH, in \cite{Schlumprecht-Ph}, Schlumprecht constructed a compact spaces $X$ without non-trivial convergent sequences such that  the Banach space $C(X)$ has the $(GP)$ property.
\end{proof}
%}

\end{document}